\documentclass[11pt,amstex]{article}

\usepackage{a4wide}
\include{a4wide.sty}
\usepackage{epsf}
\usepackage{epsfig}
\usepackage{bbm}
\usepackage{amsmath,amssymb,mathrsfs}
\usepackage{amsthm}
\usepackage{lipsum} % for filler text
\date{}
\begin{document}
\title {Reverse inequalities for the Berezin number of operators
\footnote{This project was funded by the National Plan for Science, Technology and Innovation (MAARIFAH), King Abdulaziz City for Science and Technology, Kingdom of Saudi Arabia, award number 13-MAT1276-02.
\newline
Authors' address: \ Department of Mathematics, College of Science, King Saud University,
P.O. Box 2455, Riyadh 11451,  Saudi Arabia.
\newline
Authors' Emails: mgarayev@ksu.edu.sa\ ; \ hguediri@ ksu.edu.sa ; \ najla@ksu.edu.sa}}

\author{Mubariz Garayev; Hocine Guediri and Najla Altwaijry}
\maketitle
\newtheorem{theorem}{Theorem}[section]
\newtheorem{lemma}{Lemma}[section]
\newtheorem{remark}{Remark}[section]
\newtheorem{proposition}{Proposition}[section]
\newtheorem{definition}{Definition}[section]
\newtheorem{corollary}{Corollary}[section]
\newtheorem{example}{Example}[section]
\newcommand{\BB}{\ensuremath{\mathfrak{B}}}
\newcommand{\RR}{\ensuremath{\mathfrak{R}}}
\newcommand{\LL}{\ensuremath{\mathfrak{L}}}
\newcommand{\HH}{\ensuremath{\mathfrak{H}}}
\newcommand{\eSS}{\ensuremath{\mathfrak{S}}}
\newcommand{\DD}{\ensuremath{\mathfrak{D}}}
\newcommand{\WW}{\ensuremath{\mathfrak{W}}}
\newcommand{\D}{\ensuremath{\mathbb{D}}}
\newcommand{\V}{\ensuremath{\mathbb{V}}}
\newcommand{\W}{\ensuremath{\mathbb{W}}}
\newcommand{\F}{\ensuremath{\mathbb{F}}}
\newcommand{\Q}{\ensuremath{\mathbb{Q}}}
\newcommand{\U}{\ensuremath{\mathbb{U}}}
\newcommand{\A}{\ensuremath{\mathbb{A}}}
\newcommand{\R}{\ensuremath{\mathbb{R}}}
\newcommand{\N}{\ensuremath{\mathbb{N}}}
\newcommand{\B}{\ensuremath{\mathbb{B}}}
\newcommand{\C}{\ensuremath{\mathbb{C}}}
\newcommand{\eH}{\ensuremath{\mathbb{H}}}
\newcommand{\Z}{\ensuremath{\mathbb{Z}}}
\newcommand{\K}{\ensuremath{\mathbb{K}}}
\newcommand{\T}{\ensuremath{\mathbb{T}}}
\newcommand{\M}{\ensuremath{\mathbb{M}}}
\newcommand{\X}{\ensuremath{\mathbb{X}}}
\newcommand{\Y}{\ensuremath{\mathbb{Y}}}
\newcommand{\es}{\ensuremath{\mathbb{S}}}
\newcommand{\E}{\ensuremath{\mathbb{E}}}
\newcommand{\el}{\ensuremath{\mathbb{L}}}
\newcommand{\AAA}{\ensuremath{\mathscr{A}}}
\newcommand{\BBB}{\ensuremath{\mathscr{B}}}
\newcommand{\HHH}{\ensuremath{\mathscr{H}}}
\newcommand{\LLL}{\ensuremath{\mathscr{L}}}
\newcommand{\WWW}{\ensuremath{\mathscr{W}}}
\newcommand{\KKK}{\ensuremath{\mathscr{K}}}
\newcommand{\RRR}{\ensuremath{\mathscr{R}}}
\newcommand{\CCC}{\ensuremath{\mathscr{C}}}
\newcommand{\UUU}{\ensuremath{\mathscr{U}}}
\newcommand{\MMM}{\ensuremath{\mathscr{M}}}
\newcommand{\SSS}{\ensuremath{\mathscr{S}}}
\newcommand{\pr}{\bf Proof.}
\newcommand{\epr}{\ensuremath{\blacksquare}}
\numberwithin{equation}{section}
\begin{abstract}
For a bounded linear operator $A$ on a reproducing kernel Hilbert space $\HHH(\Omega)$, with normalized reproducing kernel $\widehat{k}_{\lambda}=\frac{k_{\lambda}}{\lVert k_{\lambda}\lVert}$, the Berezin symbol, Berezin number and Berezin norm are defined respectively by 
 $\widetilde{A}(\lambda)=\langle A\widehat{k}_{\lambda},\widehat{k}_{\lambda}\rangle$, 
 $ber(A)=\sup_{\lambda\in\Omega}\left|\widetilde{A}(\lambda)\right|$  and
 $\left\|A\right\|_{ber}=\sup_{\lambda\in\Omega}\left\|A\widehat{k}_{\lambda}\right\|$. A straightforward comparison between these characteristics yields the inequalities $ber(A)\leq\left\|A\right\|_{ber}\leq\lVert A\lVert$. In this paper, we prove further  inequalities relating them, and give special care  to the corresponding reverse inequalities. In particular, we refine the first one of the above inequalities, namely we prove that \ $ber(A)\leq\left( \left\|A\right\|_{ber}^{2}-\inf_{\lambda\in\Omega}\left\lVert (A-\widetilde{A}(\lambda))\widehat{k}_{\lambda}\right\lVert^{2}\right) ^{\frac{1}{2}}.$
\end{abstract}
\noindent{{\bf Key Words:} {\sf  Berezin symbol, Berezin number, numerical radius, positive operator, 
hyponormal operator, invertible operator.}
\bigskip

\noindent{{\bf 2020 Math Subject Classification:} {\sf Primary: 47A30 ; Secondary:  47B20.}}
%%%%%%%%%%%%%%%%%%%%%%%%%%%%%%%%%%%%%%%%%%%%%%%%%%%%%%%%%%%%%%%%%%%%%%%%%%%%%%%
%%%%%%%%%%%%%%%%%%%%%%%%%%%%%%%%%%%%%%%%%%%%%%%%%%%%%%%%%%%%%%%%%%%%%%%%%%%%%%%
%%%%%%%%%%%%%%%%%%%%%%%%%%%%%%%%%%%%%%%%%%%%%%%%%%%%%%%%%%%%%%%%%%%%%%%%%%%%%%%
\section{Introduction}
%%%%%%%%%%%%%%%%%%%%%%%%%%%%%%%%%%%%%%%%%%%%%%%%%%%%%%%%%%%%%%%%%%%%%%%%%%%%%%%
%%%%%%%%%%%%%%%%%%%%%%%%%%%%%%%%%%%%%%%%%%%%%%%%%%%%%%%%%%%%%%%%%%%%%%%%%%%%%%%
%%%%%%%%%%%%%%%%%%%%%%%%%%%%%%%%%%%%%%%%%%%%%%%%%%%%%%%%%%%%%%%%%%%%%%%%%%%%%%%
\noindent A reproducing kernel Hilbert space (RKHS) is a Hilbert space $\HHH=\HHH(\Omega)$ of complex 
valued functions on a (non-empty) set $\Omega$ with the property that the evaluation functional 
$f\rightarrow f(\lambda)$ is continuous on $\HHH$ for every $\lambda\in\Omega$. Then the Riesz 
representation theorem ensures the existence of  a unique 
element $k_{\lambda}\in\HHH$,  for each $\lambda\in\Omega$, such that
\begin{equation}
\label{equ(1.2)}
f(\lambda)=\langle f,k_{\lambda}\rangle\quad\text{for all}\quad f\in\HHH.
\end{equation}
The function $k_{\lambda}$, $\lambda\in\Omega$, is called the reproducing kernel of $\HHH$. If 
$\{ e_{n}\}$ is an orthonormal basis for a RKHS $\HHH$, then its reproducing kernel is given by 
$k_{\lambda}(z)=\sum_{n}\overline{e_{n}(\lambda})e_{n}(z)$; see Aronzajn \cite{Aronzajn} and Saitoh and 
Sowano \cite{Saitoh}. For $\lambda\in\Omega$, let $\widehat{k}_{\lambda}=\frac{k_{\lambda}}{\lVert k_{\lambda}\lVert}$ be
the normalized reproducing kernel of $\HHH$. For a bounded linear operator $A\in\cal{B}(\HHH)$,  (with $\cal{B}(\HHH)$
being the Banach algebra of all bounded linear operators on $\HHH$), the function 
$\widetilde{A}:\Omega\rightarrow \C$ defined by 
\begin{equation}
\label{equ(1.3)}
\widetilde{A}(\lambda):=\langle A\widehat{k}_{\lambda},\widehat{k}_{\lambda}\rangle
\end{equation}
is the Berezin symbol of $A$, which  was first introduced by Berezin \cite{Berezin1,Berezin2}. 
The Berezin set and the Berezin number of the operator $A$ are defined, respectively, by  (see \cite{Karaev2,Karaev4}):
\begin{equation}
\label{equ(1.3.1)}
Ber(A):=\text{Range}(\widetilde{A})=\left\{\widetilde{A}(\lambda):\lambda\in\Omega\right\},
\end{equation}
and 
\begin{equation}
\label{equ(1.3.2)}
ber(A)=\sup_{\lambda\in\Omega}\left|\widetilde{A}(\lambda)\right|.
\end{equation}
It is clear that the Berezin symbol  $\widetilde{A}$ is a bounded function on $\Omega$ and that $Ber(A)\subset W(A)$ and 
$ber(A)\leq w(A)$ for all $A\in\cal{B}(\HHH)$, where
\begin{equation}
\label{equ(1.4)}
W(A):=\left\{\left\langle Ax,x\right\rangle: x\in\HHH\quad\text{and}\quad\left\|x\right\|=1\right\}
\end{equation}
is the numerical range of the operator $A$ and
\begin{equation}
\label{equ(1.5)}
w(A):=\sup\{\left|\left\langle Ax,x\right\rangle\right|: x\in\HHH\quad \text{and}\quad\left\|x\right\|=1\}
\end{equation}
is the numerical radius of $A$. It is well known that there are concrete examples of operators for which $Ber(A)$ is a proper subset of
$W(A)$  and $ber(A)< w(A)$; and others satisfying $\overline{Ber(A)}=\sigma(A)$,  $Ber(A)=W(A)$ and  $ber(A)=w(A)=\lVert A\rVert$, (see the first author's paper  \cite{Karaev2}). The Berezin number of an operator $A$ 
satisfies the following properties:
\begin{itemize}
\item[(i)] $ber(A)\leq \lVert A\lVert$.
\item[(ii)] $ber(\alpha A)=|\alpha| ber( A)$ for all $\alpha\in\C$.
\item[(iii)] $ber(A+B)\leq ber(A)+ber(B)$.
\end{itemize}
Notice that, in general, the Berezin number does not define a norm. However, if $\HHH$ is a $RKHS$ 
of analytic functions, (for instance  on the unit disc $\D=\{z\in\C: |z|<1\}$), then $ber(A)$ defines a norm on 
$\cal{B}(\HHH(\D))$; which follows from the following lemma (see, for instance, Zhu \cite{Zhu})
\begin{lemma}
\label{lem(1.1)}
 Let $\HHH=\HHH(\D)$ be a RKHS of analytic functions on $\D$, and let $A\in\cal{B}(\HHH)$ be an operator. 
 Then the Berezin symbol $\widetilde{A}$ uniquely defines the operator $A$, i.e., $A=0$ if and only if 
$\widetilde{A}=0$.
\end{lemma}
\noindent Now, for any operator $A$ on the RKHS $\HHH=\HHH(\Omega)$,  let us define the Berezin norm of the operator $A$ by
\begin{equation}
\label{equ(1.6)}
\left\|A\right\|_{ber}:=\sup_{\lambda\in\Omega}\left\|A\widehat{k}_{\lambda}\right\|_{\HHH}.
\end{equation}
Clearly  $\left\|A\right\|_{ber}$ shares the properties  (i)-(iii) with $ber(A)$. Also, since the family $\{k_{\lambda}:\lambda\in\Omega\}$ is complete in $\HHH$, 
it is elementary to verify that $\left\|A\right\|_{ber}=0$ if and only if $A=0$. So, 
these properties together mean that $\left\|A\right\|_{ber}$ is a norm in $\cal{B}(\HHH)$. Clearly 
$ber(A)\leq\left\|A\right\|_{ber}$ for any $A\in\cal{B}(\HHH)$. However, it is known that in the case of the unit disk $\D$ these two 
new operator norms are not equivalent norms with respect to the usual operator norm 
$\lVert A\lVert:=\sup\{\lVert Ax\lVert: x\in\HHH\quad\text{and}\quad\lVert x\lVert=1\}$. Namely, 
Engli\v{s} \cite{Englis1} proved that 
\begin{equation}
\label{equ(1.7)}
\lVert T_{f}\lVert\leq C\sup_{z\in\D}|\widetilde{T_{f}}(z)|=C\ ber(T_{f}), \ \forall{f}\in L^{\infty}(\D,dm_{2}),
\end{equation}
can not hold for any constant $C>0$, where $T_{f}$ is the Toeplitz operator on the Bergman Hilbert space $L_{a}^{2}=L_{a}^{2}(\D)$ and 
$dm_{2}$ is the usual normalized area measure on $\D$. Later, Nazarov showed the inequality
(see Miao and Zheng \cite{Miao} Section 6):
\begin{equation}
\label{equ(1.8)}
\lVert T_{f}\lVert\leq C\lVert T_{f}\lVert_{ber}, \ \ \forall{f}\in L^{\infty}(\D,dm_{2}),
\end{equation}
can not hold for any constant $C>0$. These results show that in general there is no  universal constants 
$C_{1}, C_{2}>0$ such that $\lVert A\lVert\leq C_{1} ber(A)$ and 
$\lVert A\lVert\leq C_{2}\lVert A\lVert_{ber}$. 

\noindent Dragomir \cite{Dragomir1,Dragomir2} obtained some elegant reverse inequalities related to the
classical numerical radius power inequality
\begin{equation}
\label{equ(1.12)}
w^{2}(A)\leq w(A^{2})+\inf_{\lambda\in\C}\lVert A-\lambda I\lVert^{2}, \ \mbox{ for }  A\in{\cal{B}}(\HHH({\Omega})).
\end{equation}
In this paper, by using some ideas of \cite{Dragomir1,Dragomir2}, we prove several reverse 
inequalities involving $ber(A)$ and $\lVert A\lVert_{ber}$. In particular, we prove some analogue of the
inequality (\ref{equ(1.12)}) for the Berezin number of operators. We also 
discuss some problems related to invertible operators and  hyponormal operators on a RKHS.
%%%%%%%%%%%%%%%%%%%%%%%%%%%%%%%%%%%%%%%%%%%%%%%%%%%%%%%%%%%%%%%%%%%%%%
%%%%%%%%%%%%%%%%%%%%%%%%%%%%%%%%%%%%%%%%%%%%%%%%%%%%%%%%%%%%%%%%%%%%%%
%%%%%%%%%%%%%%%%%%%%%%%%%%%%%%%%%%%%%%%%%%%%%%%%%%%%%%%%%%%%%%%%%%%%%%
\section{Reverse Berezin number and Berezin norm inequalities for  invertible operators}
%%%%%%%%%%%%%%%%%%%%%%%%%%%%%%%%%%%%%%%%%%%%%%%%%%%%%%%%%%%%%%%%%%%%%%
%%%%%%%%%%%%%%%%%%%%%%%%%%%%%%%%%%%%%%%%%%%%%%%%%%%%%%%%%%%%%%%%%%%%%%
%%%%%%%%%%%%%%%%%%%%%%%%%%%%%%%%%%%%%%%%%%%%%%%%%%%%%%%%%%%%%%%%%%%%%%
\noindent In this section, we prove some new reverse inequalities for the Berezin number and the Berezin norm  of 
two operators $A, B$ on $\HHH(\Omega)$ with invertible $B$. Note that similar results for 
$\lVert A\lVert$ and $w(A)$, are proved by Dragomir \cite{Dragomir1}. 
%==================================================================================================
\begin{proposition}
\label{pro(2.1)}
Let $\HHH=\HHH(\Omega)$ be a RKHS, and let $A,B\in{\cal{B}}(\HHH)$ be two operators, where $B$ 
is invertible,  satisfying  the following inequality  for a given $r>0$:
\begin{equation}
\label{equ(2.1)}
\lVert A-B\lVert_{ber}\leq r.
\end{equation}
Then, we have
\begin{equation}
\label{equ(2.2)}
\lVert A\lVert_{ber}\leq \lVert B^{-1}\lVert\left[ber(B^{*}A)+\frac{1}{2}r^{2} \right] .
\end{equation}
\end{proposition}
\begin{proof}
Clearly  (\ref{equ(2.1)}) is equivalent to the inequality 
\begin{equation}
\label{equ(2.3)}
\langle (A-B)\widehat{k}_{\lambda}, (A-B)\widehat{k}_{\lambda}\rangle\leq r^{2},
\end{equation}
which can be rephrased as
\begin{equation}
\label{equ(2.4)}
\lVert A\widehat{k}_{\lambda}\rVert^{2}+\lVert B\widehat{k}_{\lambda}\rVert^{2} 
\leq 2 Re\langle B^{*}A\widehat{k}_{\lambda},\widehat{k}_{\lambda}\rangle +r^{2},
\end{equation}
for all $\lambda\in\Omega$. Since $B$ is invertible, we obtain
\begin{equation}
\label{equ(2.5)}
\lVert Bx\rVert^{2}\geq\frac{1}{\lVert B^{-1}\rVert^{2}}\lVert x\rVert^{2}
\end{equation}
for all $x\in\HHH(\Omega)$. In particular, for $x=\widehat{k}_{\lambda}$, we have that
$
\lVert B\widehat{k}_{\lambda}\rVert^{2}\geq\frac{1}{\lVert B^{-1}\rVert^{2}}
$
for all $\lambda\in\Omega$. Now, by considering that 
$ Re\langle B^{*}A\widehat{k}_{\lambda},\widehat{k}_{\lambda}\rangle
\leq |\langle B^{*}A\widehat{k}_{\lambda},\widehat{k}_{\lambda}\rangle|=|\widetilde{B^{*}A(\lambda)}|$, 
we get from (\ref{equ(2.4)}) that
\begin{equation}
\label{equ(2.6)}
\lVert A\widehat{k}_{\lambda}\rVert^{2}+\frac{1}{\lVert B^{-1}\rVert^{2}}\leq
2|\widetilde{B^{*}A(\lambda)}|+r^{2}
\end{equation}
for all $\lambda\in\Omega$. Then taking the supremum over $\lambda\in\Omega$ in (\ref{equ(2.6)}), we obtain
\begin{equation}
\label{equ(2.7)}
\lVert A\rVert^{2}_{ber}+\frac{1}{\lVert B^{-1}\rVert^{2}}\leq
2ber(B^{*}A)+r^{2}.
\end{equation}
By the elementary geometric-arithmetic mean inequality, from (\ref{equ(2.7)}) we see that 
\begin{equation}
\label{equ(2.8)}
\frac{2\lVert A\rVert_{ber}}{\lVert B^{-1}\rVert}\leq
2ber(B^{*}A)+r^{2},
\end{equation}
which implies the desired result.
\end{proof}
\noindent In what follows, we will use the short notation $A-\mu$ instead of $A-\mu I$,  
where $I$ is the identity operator on $\HHH(\Omega)$. First, observe that we have the following consequence of Proposition \ref{pro(2.1)}:
\begin{corollary}
\label{cor(2.1)}
For an operator $A\in{\cal{B}}(\HHH)$, we have
\begin{itemize}
\item[(i)] $0\leq \lVert A\rVert_{ber}-ber(A)\leq
\frac{1}{2|\mu|}r^{2}$ \mbox{ provided that }\ $\lVert A-\mu\rVert_{ber}\leq r$.
\item[(ii)]
$\lVert A\lVert_{ber}\leq \lVert A^{-1}\lVert\left[ber(A^{2})+\frac{1}{2|\mu|}r^{2} \right]$ 
  \mbox{ provided that }\  $\lVert A-\mu A^{*}\rVert_{ber}\leq r$, $\mu\neq 0$.
\end{itemize}
\end{corollary}
\noindent It can be easily seen from the proof of Proposition \ref{pro(2.1)} that the invertibility 
condition of the operator $B$ can be replaced by the condition that
\begin{equation}
\label{equ(2.9)}
\widetilde{|B|^{2}}(\lambda)\geq C
\end{equation}
for all $\lambda\in\Omega$ and for some $C>0$, where $|B|:=(B^{*}B)^{\frac{1}{2}}
=\sqrt{B^{*}B}$ denotes the modulus (positive part) of the operator $B$. Namely, we can state without proof the following 
proposition.
%================================================================================================
\begin{proposition}
\label{pro(2.2)}
Let $A, B\in{\cal{B}}(\HHH)$ be such that (\ref{equ(2.1)}) holds, with  $B$ satisfying inequality (\ref{equ(2.9)}). Then  
\begin{equation}
\label{equ(2.10)}
\lVert A\lVert_{ber}=\sqrt{ber(|A|^{2})}\leq \frac{1}{\sqrt C}\left[ber(B^{*}A)+\frac{1}{2}r^{2} \right].
\end{equation}
\end{proposition}
%=================================================================================
\begin{proposition}
\label{pro(2.3)}
Let $A,B\in{\cal{B}}(\HHH({\Omega}))$ be two operators satisfying (\ref{equ(2.1)}) and suppose that $B$ 
is invertible. Then, we have
\begin{equation}
\label{equ(2.13)}
\lVert A\lVert_{ber}\lVert B\rVert\leq ber(B^{*}A)+\frac{1}{2}\left[ r^{2} 
+\lVert B\rVert^{2}-\lVert B^{-1}\rVert^{-2}\right] .
\end{equation}
\end{proposition}
\begin{proof}
As in the proof of Proposition \ref{pro(2.1)}, the condition  (\ref{equ(2.1)}) is equivalent to (\ref{equ(2.4)}), which is in turn
equivalent to
\begin{equation}
\label{equ(2.15)}
\lVert A\widehat{k}_{\lambda}\rVert^{2}+ \lVert B\rVert^{2}
\leq 2 Re\langle B^{*}A\widehat{k}_{\lambda},\widehat{k}_{\lambda}\rangle +r^{2}+
\lVert B\rVert^{2}-\lVert B\widehat{k}_{\lambda}\rVert^{2}.
\end{equation}
Since $ Re\langle B^{*}A\widehat{k}_{\lambda},\widehat{k}_{\lambda}\rangle
\leq |\langle B^{*}A\widehat{k}_{\lambda},\widehat{k}_{\lambda}\rangle|$, 
$\lVert B\widehat{k}_{\lambda}\lVert^{2}\geq\frac{1}{\lVert B^{-1}\lVert^{2}}$ and 
$\lVert A\widehat{k}_{\lambda}\lVert^{2}+\lVert B\lVert^{2}\geq
2\lVert B\lVert \ \lVert A\widehat{k}_{\lambda}\lVert$, 
for all $\lambda\in\Omega$, using (\ref{equ(2.15)}) we get that
\begin{equation}
\label{equ(2.16)}
2\lVert B\lVert \ \lVert A\widehat{k}_{\lambda}\lVert\leq 
2|\langle B^{*}A\widehat{k}_{\lambda},\widehat{k}_{\lambda}\rangle|+r^{2}+
\lVert B\rVert^{2}-\lVert B^{-1}\rVert^{-2}
\end{equation}
for all $\lambda\in\Omega$. Taking the supremum over $\lambda\in\Omega$, 
we deduce the required result (\ref{equ(2.13)}).
\end{proof}
\noindent Note that if, in Proposition \ref{pro(2.3)}, we choose $B=\mu A^{*}$, $\mu\neq 0$,  and
$A$ is invertible, then we get 
\begin{equation}
\label{equ(2.17)}
\lVert A\lVert_{ber}^{2}-ber(A^{2})\leq\frac{1}{2}\left[ \frac{r^{2}}{|\mu|}+ 
|\mu|\left( \lVert A\rVert^{2}-\lVert A^{-1}\rVert^{-2}\right) \right],
\end{equation}
provided that $\lVert A-\mu A^{*}\rVert\leq r$.\\
\noindent The following result can be  proved using the same argument as in the proof of Proposition \ref{pro(2.3)}.
%=================================================================================
\begin{proposition}
\label{pro(2.4)}
Let $A,B\in{\cal{B}}(\HHH({\Omega}))$. If $B$ 
is invertible and, for  $r>0$, we have
\begin{equation}
\label{equ(2.18)}
\lVert A-B\lVert_{ber}\leq r<\lVert B\rVert,
\end{equation}
then 
\begin{equation}
\label{equ(2.19)}
\lVert A\lVert_{ber}\leq\frac{1}{\sqrt{\lVert B\lVert^{2}-r^{2}}}\left[ ber(B^{*}A)+\frac{1}{2} \left(  
\lVert B\rVert^{2}-\lVert B^{-1}\rVert^{-2}\right) \right] .
\end{equation}
\end{proposition}
%======================================================================
\begin{remark}
\label{rem(2.1)}
\begin{itemize}
\item[(a)] The result of  Proposition \ref{pro(2.4)} is of particular interest. Indeed, if we choose $B=\mu I$
 with $|\mu|>r$, then (\ref{equ(2.18)}) is obviously fulfilled and by  (\ref{equ(2.19)}) we get
 \begin{equation}
 \label{equ(2.20)}
 \lVert A\lVert_{ber}\leq\frac{ber(A)}{\sqrt{1-\left( \frac{r}{|\mu|}\right)^{2}}},
 \end{equation}
 provided that $\lVert A-\mu I\rVert\leq r$.
\item[(b)] On the other hand, if  we choose $B=\mu A^{*}$ with
$\lVert A\lVert\geq\frac{r}{|\mu|}$ ($\mu\neq 0$), then by  (\ref{equ(2.19)}) we get
\begin{equation}
\label{equ(2.21)}
\lVert A\lVert_{ber}\leq\frac{1}{\sqrt{\lVert A\lVert^{2}-\left( \frac{r}{|\mu|}\right)^{2}}}
\left[ ber(A^{2})+\frac{|\mu|}{2} \left(  
\lVert A\rVert^{2}-\lVert A^{-1}\rVert^{-2}\right) \right],
\end{equation}
 provided that $\lVert A-\mu A^{*}\rVert\leq r$.
\end{itemize}
\end{remark}
%================================================================
\begin{theorem}
\label{thm(2.1)}
Let $A,B\in{\cal{B}}(\HHH)$. If  $B$ 
is invertible such that $\lVert A-B\lVert\leq r$, for $r>0$, and
\begin{equation}
\label{equ(2.22)}
\frac{1}{\sqrt{r^{2}+1}}\leq\lVert B^{-1}\lVert<\frac{1}{r}.
\end{equation}
Then 
\begin{equation}
\label{equ(2.23)}
\lVert A\lVert_{ber}^{2}\leq ber^{2}(B^{*}A)+2ber(B^{*}A)
\frac{\lVert B^{-1}\lVert-\sqrt{1-r^{2}\lVert B^{-1}\lVert^{2}}}{\lVert B^{-1}\lVert}.
\end{equation}
\end{theorem}
\begin{proof}
Let $\lambda\in\Omega$. Then by (\ref{equ(2.6)}) we have
\begin{equation}
\label{equ(2.24)}
\lVert A\widehat{k}_{\lambda}\rVert^{2}+\frac{1}{\lVert B^{-1}\rVert^{2}}\leq
2|\langle B^{*}A\widehat{k}_{\lambda},\widehat{k}_{\lambda}\rangle|+r^{2},
\end{equation}
and since $\frac{1}{\lVert B^{-1}\lVert^{2}}-r^{2}>0$, we can conclude that
$|\langle B^{*}A\widehat{k}_{\lambda},\widehat{k}_{\lambda}\rangle|>0$, and thus  obtain
\begin{equation}
\label{equ(2.25)}
\frac{\lVert A\widehat{k}_{\lambda}\rVert^{2}}{|\langle B^{*}A\widehat{k}_{\lambda},\widehat{k}_{\lambda}\rangle|}
\leq 2+\frac{r^{2}}{|\langle B^{*}A\widehat{k}_{\lambda},\widehat{k}_{\lambda}\rangle|}
-\frac{1}{\lVert B^{-1}\rVert^{2}|\langle B^{*}A\widehat{k}_{\lambda},\widehat{k}_{\lambda}\rangle|}.
\end{equation}
Subtracting $|\langle B^{*}A\widehat{k}_{\lambda},\widehat{k}_{\lambda}\rangle|$ from both sides of 
(\ref{equ(2.25)}), we obtain
\begin{eqnarray*}
\label{equ(2.26)}
\frac{\lVert A\widehat{k}_{\lambda}\rVert^{2}}{|\langle B^{*}A\widehat{k}_{\lambda},\widehat{k}_{\lambda}\rangle|}
&-&|\langle B^{*}A\widehat{k}_{\lambda},\widehat{k}_{\lambda}\rangle|
\leq 2-|\langle B^{*}A\widehat{k}_{\lambda},\widehat{k}_{\lambda}\rangle|
-\frac{1-r^{2}\lVert B^{-1}\rVert^{2}}{\lVert B^{-1}\rVert^{2}|\langle B^{*}A\widehat{k}_{\lambda},\widehat{k}_{\lambda}\rangle|}\\
&=& 2-\frac{\sqrt{1-r^{2}\lVert B^{-1}\rVert^{2}}}{\lVert B^{-1}\rVert}-
\left(\sqrt{|\langle B^{*}A\widehat{k}_{\lambda},\widehat{k}_{\lambda}\rangle|}-
\frac{\sqrt{1-r^{2}\lVert B^{-1}\rVert^{2}}}
{\lVert B^{-1}\rVert\sqrt{|\langle B^{*}A\widehat{k}_{\lambda},\widehat{k}_{\lambda}\rangle|}}
\right)^{2}\\
&\leq & 2\left( \frac{\lVert B^{-1}\rVert-\sqrt{1-r^{2}\lVert B^{-1}\rVert^{2}}}
{\lVert B^{-1}\rVert}
\right),
\end{eqnarray*}
which gives 
\begin{equation}
\label{equ(2.27)}
\lVert A\widehat{k}_{\lambda}\rVert^{2}\leq|\langle B^{*}A\widehat{k}_{\lambda},\widehat{k}_{\lambda}\rangle|^{2}
+2|\langle B^{*}A\widehat{k}_{\lambda},\widehat{k}_{\lambda}\rangle|
\frac{\lVert B^{-1}\rVert-\sqrt{1-r^{2}\lVert B^{-1}\rVert^{2}}}{\lVert B^{-1}\rVert}
\end{equation}
Notice that  (\ref{equ(2.22)}) guaranties the positivity of the nominator of the fraction in the right hand side of (\ref{equ(2.27)}), while taking the supremum in (\ref{equ(2.27)}) over $\lambda\in\Omega$, 
we deduce the desired inequality (\ref{equ(2.23)}).
\end{proof}
\noindent Note that for $\mu\in\C$ with $0<r\leq |\mu|\leq \sqrt{r^{2}+1}$ and $\lVert A-\mu I\lVert\leq r$,
we can state that
\begin{equation}
\label{equ(2.29)}
\lVert A\lVert_{ber}^{2}\leq |\mu|^{2}ber(A^{2})+2|\mu|
\left( 1-\sqrt{|\mu|^{2}-r^{2}}\right)ber(A).
\end{equation}
Also, if $\lVert A-A^{*}\lVert\leq r$, and $A$ is invertible with 
$\frac{1}{\sqrt{r^{2}+1}}\leq\lVert A^{-1}\rVert\leq\frac{1}{r}$, then by  (\ref{equ(2.23)}) 
we have
\begin{equation}
\label{equ(2.30)}
\lVert A\lVert_{ber}^{2}\leq ber^{2}(A^{2})+2ber(A^{2})
\frac{\lVert A^{-1}\lVert-\sqrt{1-r^{2}\lVert A^{-1}\lVert^{2}}}{\lVert A^{-1}\lVert}.
\end{equation}
%=============================================================
\begin{theorem}
\label{thm(2.2)}
Let $A,B\in{\cal{B}}(\HHH)$. If $B$ 
is invertible such that $\lVert A-B\lVert\leq r$, for $r>0$, and 
$\lVert B^{-1}\lVert<\frac{1}{r}$, then 
\begin{equation}
\label{equ(2.31)}
0\leq\lVert A\lVert_{ber}^{2}\lVert B\lVert^{2}-ber^{2}(B^{*}A)
\leq 2ber(B^{*}A)
\frac{\lVert B\lVert}{\lVert B^{-1}\lVert}
\left(\lVert B\lVert \ \lVert B^{-1}\lVert-\sqrt{1-r^{2}\lVert B^{-1}\lVert^{2}} \right).
\end{equation}
\end{theorem}
\begin{proof}
Subtracting the quantity 
$\frac{|\langle B^{*}A\widehat{k}_{\lambda},\widehat{k}_{\lambda}\rangle|}{\lVert B\lVert^{2}}$ 
from both sides of (\ref{equ(2.25)}), we obtain
\begin{eqnarray*}
\label{equ(2.32)}
0&\leq&\frac{\lVert A\widehat{k}_{\lambda}\rVert^{2}}{|\langle B^{*}A\widehat{k}_{\lambda},\widehat{k}_{\lambda}\rangle|}
-\frac{|\langle B^{*}A\widehat{k}_{\lambda},\widehat{k}_{\lambda}\rangle|}{\lVert B\lVert^{2}}
\leq 2-\frac{|\langle B^{*}A\widehat{k}_{\lambda},\widehat{k}_{\lambda}\rangle|}{\lVert B\lVert^{2}}
-\frac{1-r^{2}\lVert B^{-1}\rVert^{2}}{\lVert B^{-1}\rVert^{2}|\langle B^{*}A\widehat{k}_{\lambda},\widehat{k}_{\lambda}\rangle|}\\
&=& 2-2\frac{\sqrt{1-r^{2}\lVert B^{-1}\rVert^{2}}}{\lVert B\lVert \ \lVert B^{-1}\rVert}-
\left(\frac{\sqrt{|\langle B^{*}A\widehat{k}_{\lambda},\widehat{k}_{\lambda}\rangle|}}{\lVert B\lVert}-
\frac{\sqrt{1-r^{2}\lVert B^{-1}\rVert^{2}}}
{\lVert B^{-1}\rVert\sqrt{|\langle B^{*}A\widehat{k}_{\lambda},\widehat{k}_{\lambda}\rangle|}}
\right)^{2}\\
&\leq & 2\left( \frac{\lVert B\lVert \ \lVert B^{-1}\rVert-\sqrt{1-r^{2}\lVert B^{-1}\rVert^{2}}}
{\lVert B\lVert \ \lVert B^{-1}\rVert}
\right),
\end{eqnarray*}
which is equivalent to
\begin{equation}
\label{equ(2.33)}
0\leq\lVert A\lVert_{ber}^{2}\lVert B\lVert^{2}-|\widetilde{B^{*}A}(\lambda)|^{2}
\leq 
2\frac{\lVert B\lVert}{\lVert B^{-1}\lVert}|\widetilde{B^{*}A}(\lambda)|
\left(\lVert B\lVert \ \lVert B^{-1}\lVert-\sqrt{1-r^{2}\lVert B^{-1}\lVert^{2}} \right),
\end{equation}
for all $\lambda\in\Omega$. The inequality (\ref{equ(2.33)}) also shows that
$\lVert B\lVert \ \lVert B^{-1}\lVert\geq\sqrt{1-r^{2}\lVert B^{-1}\lVert^{2}}$, and then, by
(\ref{equ(2.33)}), we get 
\begin{equation}
\label{equ(2.34)}
\lVert A\lVert_{ber}^{2}\lVert B\lVert^{2}
\leq |\widetilde{B^{*}A}(\lambda)|^{2}+
2\frac{\lVert B\lVert}{\lVert B^{-1}\lVert}|\widetilde{B^{*}A}(\lambda)|
\left(\lVert B\lVert \ \lVert B^{-1}\lVert-\sqrt{1-r^{2}\lVert B^{-1}\lVert^{2}} \right),
\end{equation}
for all $\lambda\in\Omega$. Thus, taking the supremum in (\ref{equ(2.34)}) we 
deduce the desired inequality (\ref{equ(2.31)}).
\end{proof}
\noindent As mentioned above, Theorem \ref{thm(2.2)} is of particular interest since
putting $B=\mu I$ with $|\mu|\geq r>0$ and assuming that  $\lVert A-\mu I\lVert\leq r$, then by 
inequality (\ref{equ(2.34)}) we get
\begin{equation}
\label{equ(2.35)}
0\leq\lVert A\lVert_{ber}^{2}-ber(A^{2})\leq 2|\mu|ber(A)
\left( 1-\sqrt{1-\left( \frac{r}{|\mu|}\right)^{2}}\right).
\end{equation}
Also, if $A$ is invertible and $\lVert A-\mu A^{*}\lVert\leq r$ and
$\lVert A^{-1}\lVert\leq\frac{|\mu|}{r}$, then by (\ref{equ(2.31)}) we obtain
\begin{equation}
\label{equ(2.36)}
0\leq\lVert A\lVert_{ber}^{4}-ber^{2}(A^{2})
\leq 2|\mu|ber(A^{2})
\frac{\lVert A\lVert}{\lVert A^{-1}\lVert}
\left(\lVert A\lVert \ \lVert A^{-1}\lVert-\sqrt{1-\frac{r^{2}}{|\mu|^{2}}\lVert A^{-1}\lVert^{2}} \right).
\end{equation}
%%%%%%%%%%%%%%%%%%%%%%%%%%%%%%%%%%%%%%%%%%%%%%%%%%%%%%%%%%%%%%%%%%%%%%
%%%%%%%%%%%%%%%%%%%%%%%%%%%%%%%%%%%%%%%%%%%%%%%%%%%%%%%%%%%%%%%%%%%%%%
%%%%%%%%%%%%%%%%%%%%%%%%%%%%%%%%%%%%%%%%%%%%%%%%%%%%%%%%%%%%%%%%%%%%%%
\section{Berezin number and Berezin norm inequalities for some operators}
%%%%%%%%%%%%%%%%%%%%%%%%%%%%%%%%%%%%%%%%%%%%%%%%%%%%%%%%%%%%%%%%%%%%%%
%%%%%%%%%%%%%%%%%%%%%%%%%%%%%%%%%%%%%%%%%%%%%%%%%%%%%%%%%%%%%%%%%%%%%%
%%%%%%%%%%%%%%%%%%%%%%%%%%%%%%%%%%%%%%%%%%%%%%%%%%%%%%%%%%%%%%%%%%%%%%
\noindent A bounded operator $T$ acting on a complex infinite dimensional Hilbert 
space $H$ is said to be normal if $T^{*}T=TT^{*}$; it is said to be positive if $\langle Tx,x\rangle\geq 0$ 
for all $x\in H$ (in this case we will write $T\geq 0$); and it is said to be hyponormal if its 
self-commutator $[T^{*},T]$ is positive, that is $T^{*}T-TT^{*}\geq 0$. It is immediate from these 
definitions that every normal operator is hyponormal and that an operator $T$ is hyponormal if 
and only if $\lVert T^{*}x\rVert\leq \lVert Tx\lVert$ for all $x\in H$. It is 
also obvious that every nonunitary isometry $V:H\rightarrow H$\ (i.e. $V^{*}V=I$ and $VV^{*}\not=I$) is hyponormal, 
but not normal.\\
In this section, we will consider some operators, including hyponormal operators, and prove some inequalities 
for their Berezin numbers and Berezin norms.
\begin{proposition}
\label{pro(3.1)}
Let $N\in{\cal{B}}(\HHH({\Omega}))$ be a normal operator on the RKHS $\HHH(\Omega)$, 
and let $n\geq 1$ be any integer. Then 
\begin{equation}
\label{equ(3.18)}
ber\left( N^{n}\right) \leq
\begin{cases}
\lVert N^{\frac{n}{2}}\rVert_{ber}\quad\text \ \ \ \ \ \ \ \ \ \ \   {if}\quad n\quad\text{is even}, \\
\lVert N\rVert \ \lVert N^{\frac{n-1}{2}}\rVert_{ber}\quad\text \ {if}\quad n\quad\text{is odd}.
\end{cases}
 \end{equation}
\end{proposition}
\begin{proof}
For any fixed integer $k\geq 0$, we have 
\begin{eqnarray*}
\label{equ(3.1)}
|\langle N^{2k}\widehat{k}_{\lambda},\widehat{k}_{\mu}\rangle|&=&
|\langle N^{k}\widehat{k}_{\lambda},N^{k*}\widehat{k}_{\mu}\rangle|\leq
\lVert N^{k}\widehat{k}_{\lambda}\lVert \ \lVert N^{k*}\widehat{k}_{\mu}\lVert \\
&\leq&\sup_{\lambda\in\Omega}\lVert N^{k}\widehat{k}_{\lambda}\lVert \ 
\sup_{\mu\in\Omega}\lVert N^{k}\widehat{k}_{\mu}\lVert=\lVert N^{k}\lVert^{2}_{ber}
\end{eqnarray*}
for all $\lambda,\mu\in\Omega$. So, we get
\begin{equation}
\label{equ(3.2)}
|\langle N^{2k}\widehat{k}_{\lambda},\widehat{k}_{\mu}\rangle|\leq \lVert N^{k}\lVert^{2}_{ber}
\quad\forall k\geq 0.
\end{equation}
 In particular, for $\mu =\lambda$, we obtain
\begin{equation}
\label{equ(3.3)}
ber(N^{2k})\leq \lVert N^{k}\lVert^{2}_{ber}
\quad\forall k\geq 0.
\end{equation}
For $k=1$, we have from (\ref{equ(3.3)}) that 
\begin{equation}
\label{equ(3.4)}
ber(N^{2})\leq \lVert N^{}\lVert^{2}_{ber}
\end{equation}
for any normal operator $N$ in ${\cal{B}}(\HHH({\Omega})$.\\
\noindent On the other hand, we have
\begin{eqnarray*}
\label{equ(3.5)}
|\langle N^{2k+1}\widehat{k}_{\lambda},\widehat{k}_{\mu}\rangle|&=&
|\langle N^{2k}N\widehat{k}_{\lambda},\widehat{k}_{\mu}\rangle|=
|\langle N^{k}N\widehat{k}_{\lambda},N^{k*}\widehat{k}_{\mu}\rangle|\leq
\lVert N^{k}N\widehat{k}_{\lambda}\lVert \ \lVert N^{k*}\widehat{k}_{\mu}\lVert \\
&\leq&\sup_{\lambda\in\Omega}\lVert N^{k}N\widehat{k}_{\lambda}\lVert \ 
\sup_{\mu\in\Omega}\lVert N^{k}\widehat{k}_{\mu}\lVert
\leq\lVert N\lVert \ \lVert N^{k}\lVert^{2}_{ber}
\end{eqnarray*}
for all $\lambda,\mu\in\Omega$.  In particular, by putting $\mu =\lambda$ 
we have from the last inequality that 
\begin{equation}
\label{equ(3.6)}
ber(N^{2k+1})\leq\lVert N\lVert \  \lVert N^{k}\lVert^{2}_{ber}
\quad\forall k\geq 0.
\end{equation}
Now, the desired result follows from (\ref{equ(3.3)}) and (\ref{equ(3.5)}), which proves the proposition.
\end{proof}
\noindent The following result gives the exact relationship between the Berezin number and the Berezin norm 
of a concrete orthogonal projection (normal operator).
\begin{example}
\label{exp(3.6)}
Let $S$ be the shift operator, $Sf=zf$, on the Hardy space $H^{2}(\D)$ over 
the unit disc $\D=\{z\in\C: |z|<1\}$, which consists of analytic functions on
$\D$ having the sequence of Taylor coefficients belonging to the  space 
$l^{2}$. Consider the operator $A$ defined on $H^{2}(\D)$  by
\begin{equation}
\label{equ(3.7)}
A=S(I-SS^{*})S^{*}.
\end{equation}
Then, $ber(A)=\lVert A\lVert^{2}_{ber}=\frac{1}{4}$.
\end{example}
\begin{proof}
First, note that $\widehat{k}_{\lambda}(z)=\frac{(1-|\lambda|^{2})^{\frac{1}{2}}}{1-\overline{\lambda}z}$ 
$(\lambda\in\D)$ is the normalized reproducing kernel of $H^{2}(\D)$. Then, for all $\lambda\in\D$, we have
\begin{eqnarray}
\label{equ(3.8)}
\nonumber
\lVert A\widehat{k}_{\lambda}\lVert &=&\lVert S(I-SS^{*})S^{*}\widehat{k}_{\lambda}\lVert=
\lVert (I-SS^{*})\overline{\lambda}\widehat{k}_{\lambda}\lVert\\
&=& |\lambda|\left\lVert  \frac{(1-|\lambda|^{2})^{\frac{1}{2}}}{1-\overline{\lambda}z}-
\frac{\overline{\lambda}z(1-|\lambda|^{2})^{\frac{1}{2}}}{1-\overline{\lambda}z}\right\rVert
=|\lambda|(1-|\lambda|^{2})^{\frac{1}{2}}
\end{eqnarray}
On the other hand, for all $\lambda\in\D$, we have
\begin{eqnarray}
\label{equ(3.9)}
\nonumber
\widetilde{A}(\lambda) &=&\langle S(I-SS^{*})S^{*}\widehat{k}_{\lambda},\widehat{k}_{\lambda}\rangle
=\langle (I-SS^{*})S^{*}\widehat{k}_{\lambda},S^{*}\widehat{k}_{\lambda}\rangle\\
&=& |\lambda|^{2}\langle (I-SS^{*})\widehat{k}_{\lambda},\widehat{k}_{\lambda}\rangle=
|\lambda|^{2}(1-|\lambda|^{2})
\end{eqnarray}
Thus, we get that $\widetilde{A}(\lambda)=\lVert A\widehat{k}_{\lambda}\lVert^2$ 
for all $\lambda\in\D$. This implies that one has \ $ber(A)=\lVert A\lVert^{2}_{ber}$. Also, by using that for the function 
$f(x)=x(1-x)$, ($0\leq x<1$), $f_{\max}=f(\frac{1}{2})$, we deduce from the formula 
$\widetilde{A}(\lambda)=|\lambda|^{2}(1-|\lambda|^{2})$ that 
$\sup_{\lambda\in\D}(\widetilde{A}(\lambda))=\frac{1}{4}$, that is $ber(A)=\frac{1}{4}$. 
Hence $ber(A)=\lVert A\lVert^{2}_{ber}=\frac{1}{4}$, as desired.
\end{proof}
\noindent For any hyponormal operator $T\in{\cal{B}}(\HHH({\Omega}))$ it is easy to show that:
\begin{itemize}
\item[(i)] $\lVert T^{*}\lVert_{ber}\leq\lVert T\lVert_{ber}$;
\item[(ii)] 
\begin{equation}
\label{equ(3.10)}
ber\left([ T^{*},T]\right)\leq \lVert T\lVert_{ber}^{2} .
\end{equation}
\end{itemize}
In fact, for any $\lambda\in\Omega$, we have  that 
\begin{eqnarray}
\label{equ(3.11)}
0\leq \widetilde{[ T^{*},T]}(\lambda)=\langle (T^{*}T-TT^{*})\widehat{k}_{\lambda},\widehat{k}_{\lambda}\rangle
=\lVert T\widehat{k}_{\lambda}\lVert^{2}-\lVert T^{*}\widehat{k}_{\lambda}\lVert^{2}
\end{eqnarray}
From this, it is immediate that 
$\lVert T^{*}\lVert_{ber}\leq\lVert T\lVert_{ber}$. Since $T$ is hyponormal,  we have
$\lVert T\widehat{k}_{\lambda}\lVert\geq\lVert T^{*}\widehat{k}_{\lambda}\lVert$, 
and hence we deduce from (\ref{equ(3.11)}) $\widetilde{[ T^{*},T]}(\lambda)\leq\lVert T\widehat{k}_{\lambda}\lVert^{2}$ for all $\lambda\in\Omega$. Therefore, we get \
$ber\left([ T^{*},T]\right)\leq \lVert T\lVert_{ber}^{2}$, which proves inequality 
(\ref{equ(3.10)}).\\
\noindent Our next result is about more general operators.
%=============================================================================================
\begin{theorem}
\label{thm(3.1)}
Let $A\in{\cal{B}}(\HHH({\Omega}))$. Then 
\begin{equation}
\label{equ(3.12)}
ber(A)\leq\left( \lVert A\rVert^{2}_{ber}-
\inf_{\mu\in\Omega}\left\lVert(A-\widetilde{A}(\mu))\widehat{k}_{\mu}\right\rVert^{2}\right)^{\frac{1}{2}} .
\end{equation}
\end{theorem}
\begin{proof}
It is easy to see that 
\begin{equation}
\label{equ(3.13)}
\left\lVert(A-\widetilde{A}(\lambda))\widehat{k}_{\lambda}\right\rVert^{2}=
\left\lVert A\widehat{k}_{\lambda}\right\rVert^{2}-\left|\widetilde{A}(\lambda)\right|^{2},
\quad\forall\lambda\in\Omega.
\end{equation}
Then
\begin{equation}
\label{equ(3.14)}
\left|\widetilde{A}(\lambda)\right|^{2}=\left\lVert A\widehat{k}_{\lambda}\right\rVert^{2}
-\left\lVert(A-\widetilde{A}(\lambda))\widehat{k}_{\lambda}\right\rVert^{2},
\quad\forall\lambda\in\Omega.
\end{equation}
Whence 
\begin{eqnarray}
\label{equ(3.15)}
\nonumber
\left|\widetilde{A}(\lambda)\right|^{2}&\leq&
\left(\left\lVert A\widehat{k}_{\lambda}\right\rVert^{2} -
\inf_{\mu\in\Omega}\left\lVert(A-\widetilde{A}(\mu))\widehat{k}_{\mu}\right\rVert^{2}\right)^{\frac{1}{2}}\\
&\leq&\left(\sup_{\lambda\in\Omega}\left\lVert A\widehat{k}_{\lambda}\right\rVert^{2} -
\inf_{\mu\in\Omega}\left\lVert(A-\widetilde{A}(\mu))\widehat{k}_{\mu}\right\rVert^{2}\right)^{\frac{1}{2}}
\end{eqnarray}
for all $\lambda\in\Omega$, which implies the desired inequality (\ref{equ(3.12)}).
\end{proof}
\noindent The next corollary  gives an example of operators for which $ber(A)<\lVert A\rVert_{ber}$.
\begin{corollary}
\label{cor(3.1)}
If \ $\inf_{\mu\in\Omega}\left\lVert(A-\widetilde{A}(\mu))\widehat{k}_{\mu}\right\rVert>0$, then  \
$ber(A)<\lVert A\rVert_{ber}$.
\end{corollary}
\smallskip

\noindent In conclusion, note that conditions of the type 
\begin{equation}
\label{equ(3.16)}
\left\lVert(A-\widetilde{A}(\lambda))\widehat{k}_{\lambda}\right\rVert\rightarrow 0
\quad\text{and}\quad
\left\lVert(A-\widetilde{A}(\mu))^{*}\widehat{k}_{\mu}\right\rVert\rightarrow 0
\end{equation}
as $\lambda\rightarrow\partial\Omega$ define the so-called Engli\v{s} C*-operator algebras 
which are studied, for instance, in \cite{Englis2} and \cite{Karaev1}, and are closely related 
to an  unsolved question of Engli\v{s} %\cite{Englis2} 
 [8, Question 1] for Bergman space Toeplitz operators, where, in particular, he asked the following question: \\ 
Is it true that 
\begin{equation}
\label{equ(3.17)}
T_{\varphi}\in\AAA_{B}:=\left\lbrace 
T\in{\cal{B}}(L^{2}_{a}(\D)):\left\lVert T\widehat{k}_{a,\lambda}\right\rVert^{2}-
\left|\widetilde{T}(\lambda)\right|^{2}\rightarrow 0\quad\text{radially, and similarly for }T^{*}
\right\rbrace, 
\end{equation}
for all $\varphi\in L^{\infty}(\D)$? \\ Here, $\widehat{k}_{a,\lambda}=\frac{(1-|\lambda|^2)^2}{(1-\overline{\lambda}z)^2}$ is the normalized reproducing kernel of $L^2_a(\D)$. In \cite{ Garayev4,Gurdal}, the authors discuss this question via the so-called maximal Berezin set and the 
Berezin number.
\end{document}